\documentclass{amsart}

\usepackage{graphicx}

\newtheorem{theorem}{Theorem}[section]
\newtheorem{lemma}[theorem]{Lemma}

\newtheorem{conjecture}{Conjecture}[section]
\theoremstyle{definition}
\newtheorem{definition}[theorem]{Definition}
\newtheorem{example}[theorem]{Example}

\newtheorem{corollary}{Corollary}
\theoremstyle{remark}
\newtheorem{remark}[theorem]{Remark}
\newtheorem{question}{Question}

\numberwithin{equation}{section}

\begin{document}

\title{Complete sets}

\author{Theophilus Agama}
\address{Department of Mathematics, African Institute for Mathematical science, Ghana
}
%\curraddr{Department of Mathematics and Statistics,
%{Case Western Reserve University, Cleveland, Ohio 43403}
\email{theophilus@aims.edu.gh/emperordagama@yahoo.com}

%\author{Author Two}
%\address{Mathematical Research Section, School of Mathematical Sciences,
%Australian National University, Canberra ACT 2601, Australia}

%    General info

\date{\today}



%% The correct journal style for \specialsection is all uppercase; a known bug
%% in amsart.cls prevents this, so input must be uppercase until it is fixed.
%\specialsection*{This is a Special Section Head}
%\specialsection*{THIS IS A SPECIAL SECTION HEAD}
%This is an example of a special section head%
%%%%%%%%%%%%%%%%%%%%%%%%%%%%%%%%%%%%%%%%%%%%%%%%%%%%%%%%%%%%%%%%%%%%%%%%
\footnote{
\par
}%
%%%%%%%%%%%%%%%%%%%%%%%%%%%%%%%%%%%%%%%%%%%%%%%%%%%%%%%%%%%%%%%%%%%%%%%%
.

\begin{abstract}
In this paper we introduce the concept of completeness of sets. We study this property on the set of integers. We examine how this property is preserved as we carry out various operations compatible with sets. We also introduce the problem of counting the number of complete subsets of any given set. That is,  given any interval of integers $\mathcal{H}:=[1,N]$ and letting $\mathcal{C}(N)$ denotes the complete set counting function, we establish the lower bound $\mathcal{C}(N)\gg N\log N$.
\end{abstract}

\maketitle

\section{INTRODUCTION}
The development of set theory  dates back to the days of the German mathematician George Cantor. Infact he was one of the major  pioneers of set theory and it's development,  and so, he is thought today as the major force behind it \cite{ferreiros2008labyrinth}. Today it is widely studied in many areas of mathematics, including number theory, combinatorics, computer science, algebra etc. Intuititively, a set can be thought of as a collection of well-defined objects. The objects in the set can be seen as it's members or elements. These elements do characterize and tell us more about the nature of the set in question. The elements of a set can either be finite or infinite. For example the set $\mathcal{A}:=\{2,5,9,1,-54\}$ denotes a finite set of integers, since all the elements are integers. The set of $\mathbb{R}$ of real numbers and the set $\mathbb{Z}$ of integers are examples of infinite sets. \\ \\In what follows we set  $\mathcal{A}\pm\mathcal{B}:=\{a_i\pm b_i:a_i\in \mathcal{A}~\text{and}~b_i\in \mathcal{B}\}$, $\mathcal{A\cdot B}:=\{a_ib_j:a_i\in \mathcal{A}, b_i\in \mathcal{B}\}$ and $c\cdot \mathcal{A}:=\{ca:a\in \mathcal{A}\}$, $\mathcal{A}\setminus\mathcal{B}:=\mathcal{A}-\mathcal{B}$ for finite sets of integers $\mathcal{A}$ and $\mathcal{B}$. We recall an arithmetic progression of length $n$ to be the set $\mathcal{A}$ of the form $\mathcal{A}=\{a_0, a_0+q, a_0+2q, \ldots, a_0+(n-1)q\}$. In a more special case we have the $\mathcal{A}:=\{q, 2q,\ldots, nq\}$, a homogenous arithmetic progression. For the set $\mathcal{A}:=\{a_0,a_1, \ldots, a_n\}$, we call $\mathcal{A}^{(N)}:=\{a_0', a_1', \ldots, a_n'\}$, where $a_i'=\frac{a_i-a_0}{d(\mathcal{A})}$ and where $d(\mathcal{A})=(a_0-a_0, a_1-a_0, \ldots, a_n-a_0)$ with $0=a_0'<a_1'<\ldots <a_n'$, the normal form of $\mathcal{A}$.

There are various classifications concerning set of integers. For example, the theory of multiple sets and primitive sets is very vast and rich (See \cite{nathanson2000elementary}). A set $\mathcal{F}$ can also be classed as sumfree  if the relation $a+b=c$ is not satisfied in $\mathcal{F}$, for $a,b,c\in \mathcal{F}$. In this paper, however, we study a particular class of sets of integers. 

\section{COMPLETE SETS}
In this section we introduce the concept of completeness of a set. Using various operations compatible with sets, we investigate how this property is preserved. 

\begin{definition}\label{def1}
Let $\mathcal{A}:=\{a_{1},a_2,\cdots, a_n\}$ be a finite set of elements in $\mathcal{B}$, where addition and multiplication is well defined in $\mathcal{B}$. Then $\mathcal{A}$ is said to be complete in $\mathcal{B}$ if there exists some $b\in \mathcal{B}$ such that $\prod \limits_{i=1}^{n}a_i=b\sum \limits_{i=1}^{n}a_i$.
\end{definition}
\bigskip

It follows from the above definition the nature of a complete set will depend on the set $\mathcal{B}$. If we take $\mathcal{B}:=\mathbb{R}$, then the complete set in question will be a complete set of real numbers. Again if we take $\mathcal{B}:=\mathbb{N}$, then the complete set in question will be a complete set of natural numbers. If $\mathcal{B}:=\mathbb{F}[x]$, then any complete set under $\mathcal{B}$ will be a complete set of polynomials. Let us consider the finite set natural numbers $\mathcal{P}:=\{3,5,7\}$. It is easily seen that this set is a complete set of natural numbers. Again we notice that the set $\mathcal{F}:=\{x^2, -x^2, 2x^2\}$ is a complete set of polynomials in $\mathbb{Z}[x]$. However, the set $\{4x^3,7x^3, 10x^3 \}$ is not complete in $\mathbb{Z}[x]$.  So therefore there are, if not infinitely many, complete sets under any given type of set. Every finite set of real numbers is easily seen to be complete in $\mathbb{R}$, hence the concept of completeness is not very interesting in this setting. Thus we examine this concept on the set of integers $\mathbb{Z}$, where it is very strong.

\section{COMPLETENESS IN $\mathbb{Z}$}
In this section we study the concept of completeness of finite sets of integers. In this case the set $\mathcal{B}$ in definition \ref{def1} reduces to the set of all integers. Hence we can rewrite the definition in this particular setting as follows:

\begin{definition}
Let $\mathcal{A}:=\{a_1, a_2, \ldots, a_n\}$ be a finite set of integers. Then $\mathcal{A}$ is said to be complete in $\mathbb{Z}$ if there exists some $b\in \mathbb{Z}$ such that  $\prod \limits_{i=1}^{n}a_i=b\sum \limits_{i=1}^{n}a_i$.
\end{definition}

\subsection{EXAMPLES OF COMPLETE SETS IN $\mathbb{Z}$}
\begin{enumerate}
\item [(i)] The sets $\{3,5,7\}$, $\{-2, 5,3,-1\}$, $\{1,3,2\}$, $\{3,7,11\}$ are examples of sets of integers complete in $\mathbb{Z}$.

\item [(ii)]The sets $\{3,7,9,4, 2\}$, $\{7,11, 13, 15\}$ , $\{1,18, 17,3\}$ are not complete in $\mathbb{Z}$.

\item [(iii)] The sets $\{2,4,6\}$, $\{7,14,21,28, 35\}$, $\{3,5,12\}$ are all complete in $\mathbb{Z}$.
\end{enumerate}

\subsection{PROPERTIES OF COMPLETENESS OF SETS IN $\mathbb{Z}$}
In this section we examine some properties of completeness of finite set of integers in $\mathbb{Z}$. We examine how this property is preserved as we perform various algebras compatible with sets.

\begin{theorem}
Let $\mathcal{A}_1:=\{a_1, a_2, \ldots, a_n\}$ and $\mathcal{A}_2:=\{b_1, b_2, \ldots, b_n\}$ be complete sets in $\mathbb{Z}$. Then the following remain valid:
\begin{enumerate}
\item [(i)]The prodset $\mathcal{A}_1\cdot \mathcal{A}_2$ is also complete in $\mathbb{Z}$.

\item [(ii)]The union $\mathcal{A}_1\cup \mathcal{A}_2$ is complete in $\mathbb{Z}$ provided there exist some $t\in \mathbb{Z}$ such that $a_ib_j=t(a_i+b_j)$ for each $1\leq i,j\leq n$.

\item [(iii)]Let $\mathcal{H}:=\{c_1,c_2, \ldots, c_n\}$. Then the set $\mathcal{A}_1\cup \mathcal{H}$ is complete in $\mathbb{Z}$ provided $\mathcal{A}_1\cap \mathcal{H}=\emptyset$ and $c_1+c_2+\cdots c_n=0$\label{agama}.

\item [(iv)] Let $\mathcal{H}:=\{a_1, a_2,\ldots, a_n, b_1, b_2, \ldots, b_n\}$. Then the set $\mathcal{H}\setminus \mathcal{A}_1:=\{b_1, b_2, \ldots b_n\}$ is also complete in $\mathbb{Z}$ provided for each $i=1,\ldots n$, $b_i=ta_i$ for some fixed $t\in \mathbb{Z}$.

\item [(v)]Let $\mathcal{A}:=\{d_1, d_2, \ldots, d_n\}$ be complete in $\mathbb{Z}$ and suppose  $|2A|=\frac{n(n+1)}{2}$. Then the two fold sumset $2A$ is complete in $\mathbb{Z}$ provided  $(n+1)|(d_i+d_j)$ with $i\neq j$  for some $1\leq i,j \leq n$.

\item [(vi)]The set $q\cdot \mathcal{A}_1:=\{qa_1, qa_2,\ldots, qa_n\}$ is also complete in $\mathbb{Z}$.
\end{enumerate}
\end{theorem}

\begin{proof}
$(i)$ Suppose $\mathcal{A}_1:=\{a_1,a_2,\ldots, a_n\}$ and $\mathcal{A}_2:=\{b_1, b_2, \ldots, b_n\}$. Then the prodset $\mathcal{A}_1\cdot \mathcal{A}_2:=\{a_1b_1, \ldots, a_1b_n,$ $ a_2b_1, a_2b_2, \ldots, a_2b_n, \ldots, a_nb_1, a_nb_2, \ldots, a_nb_n\}$. Now, it follows that\begin{align}\prod \limits_{i,j=1}^{n}a_ib_j&=(b_1b_2\cdots b_n)(a_1a_2\cdots a_n)\bigg((b_1b_2\cdots b_n)^{n-1}\prod \limits_{i=1}^{n}a_i^{n-1}\bigg).\nonumber 
\end{align}Since $\mathcal{A}_1$ and $\mathcal{A}_2$ are complete in $\mathbb{Z}$, it follows that $\prod \limits_{i,j=1}^{n}a_ib_j=K\bigg(\sum \limits_{i=1}^{n}a_i\bigg)\bigg(\sum \limits_{j=1}^{n}b_j\bigg)$ $\bigg((b_1b_2\cdots b_n)^{n-1}\prod \limits_{i=1}^{n}a_i^{n-1}\bigg)$, for some $K\in \mathbb{Z}$. Thus we can write \begin{align}\prod \limits_{i,j=1}^{n}a_ib_j=R\sum \limits_{i,j=1}^{n}a_ib_j,\nonumber
\end{align}where it is easily seen that $R\in \mathbb{Z}$. Hence the conclusion follows immediately.
\bigskip

$(ii)$ Suppose $\mathcal{A}_1:=\{a_1,a_2, \ldots, a_n\}$ and $\mathcal{A}_2:=\{b_1,b_2,\ldots, b_n\}$ be complete in $\mathbb{Z}$. Now the union $\mathcal{A}_1\cup \mathcal{A}_2=\{a_1, a_2, \ldots, a_n,b_1, b_2, \ldots, b_n\}$. Let us take the product of all the elements of the set, given by $(a_1a_2\cdots a_n)(b_1b_2\cdots b_n)$. Bearing in mind each of the sets is complete in $\mathbb{Z}$, it follows that $(a_1a_2\cdots a_n)(b_1b_2\cdots b_n)=R\bigg(\sum\limits_{i=1}^{n}a_i\bigg)\bigg(\sum \limits_{j=1}^{n}b_j\bigg)=R\sum \limits_{i,j=1}^{n}a_ib_j$. It follows from the hypothesis that $\sum \limits_{i,j=1}^{n}a_ib_j=S\sum \limits_{i=1}^{n}a_i+S\sum \limits_{j=1}^{n}b_j$, where $S\in \mathbb{Z}$. Hence it is easy to see that the conclusion follows immediately.
\bigskip

$(iii)$ Suppose $\mathcal{A}_1:=\{a_1, a_2, \ldots, a_n\}$ be a complete set in $\mathbb{Z}$ and let $\mathcal{H}:=\{c_1, c_2, \ldots, c_n\}$ such that $c_1+c_2+\cdots +c_n=0$. Assume the set $\mathcal{A}_1\cup \mathcal{H}:=\{a_1,a_2, \ldots, a_n, c_1, c_2, \ldots, c_n\}$ so that $\mathcal{A}_1\cap \mathcal{H}=\emptyset$. Since $\mathcal{A}_1$ is complete in $\mathbb{Z}$, we see that $a_1\cdot a_2 \cdots a_nc_1\cdot c_2\cdots c_n=R(c_1c_2\cdots c_n)(a_1+a_2\cdots a_n+c_1+c_2+\cdots +c_n)=R_1(a_1+a_2\cdots a_n+c_1+c_2+\cdots +c_n)$ and it follows that $\mathcal{A}_1\cup \mathcal{H}$ is also complete in $\mathbb{Z}$.
\bigskip

$(iv)$ Consider the set $\mathcal{H}\setminus \mathcal{A}_1:=\{b_1, b_2, \ldots, b_n\}$. Using the hypothesis and the fact that $\mathcal{A}_1$ is complete in $\mathbb{Z}$, it follows that $b_1\cdot b_2\cdots b_n=t^n(a_1\cdot a_2\cdots a_n)=t^nk(a_1+a_2+\cdots +a_n)$. Hence it follows that $\mathcal{H}\setminus \mathcal{A}_1$ is also complete in $\mathbb{Z}$.
\bigskip

$(v)$ Suppose $\mathcal{A}:=\{d_1, d_2, \ldots, d_n\}$ is complete in $\mathbb{Z}$ and let $|2\mathcal{A}|=\frac{n(n+1)}{2}$. Then the two fold sumset $2\mathcal{A}:=\{d_1+d_1, d_1+d_2, \ldots, d_1+d_n, d_2+d_2, d_2+d_3, \ldots, d_2+d_n, \ldots, d_{n-1}+d_{n-1}, d_{n-1}+d_{n}, d_{n}+d_{n}\}$. The product of the elements in $2\mathcal{A}$ is given by $2^{n}(d_{1}\cdot d_2\cdots d_{n})\prod \limits_{\substack{i,j=1\\i\neq j}}^{n}(d_i+d_j)$. Since $\mathcal{A}$ is complete in $\mathbb{Z}$ and $n+1|(d_i+d_j)$ for some $1\leq i,j \leq n$ with $i\neq j$, we have that $2^{n}(d_{1}\cdot d_2\cdots d_{n})\prod \limits_{\substack{i,j=1\\i\neq j}}^{n}(d_i+d_j)=R(n+1)\sum \limits_{i=1}^{n}d_i$ with $R\in \mathbb{Z}$, and we see that the result follows immediately.
\bigskip

$(vi)$ The result follows immediately, since $\mathcal{A}_1$ is complete in $\mathbb{Z}$.
\end{proof}

\begin{remark}
Property $(iii)$ in Theorem \ref{agama} is very important and useful for construction purposes. It tells us we only need to find a complete set of small size as we seek for a large complete set, since a larger complete set can be obtained by adding  well-balanced elements of any size we wish into the set. Again property $(v)$ in Theorem \ref{agama} informs us that if a set is complete, then the two fold sumset has a very high chance of being complete provided the size is not too small. Finally property $(ii)$ tells us that if the product of any two elements from any two complete sets, not necessarily distinct, can be controlled additively then their union will certainly be complete in $\mathbb{Z}$.
\end{remark}

\begin{theorem}
Let $\mathcal{G}$ be a finite set of integers. Then the normal form of $\mathcal{G}$, denoted $\mathcal{G}^{(N)}$ is always complete in $\mathbb{Z}$.
\end{theorem}

\begin{proof}
Consider the set $\mathcal{G}=\{a_0, a_1, \ldots, a_n\}$ of integers, where $a_0<a_1<\cdots <a_n$. The normal form of $\mathcal{G}$ is given by $\mathcal{G}^{(N)}:=\{a_0', a_1', \ldots, a_n'\}$, where $a_i'=\frac{a_i-a_0}{d(\mathcal{G})}$ and where $d(\mathcal{G})=(a_0-a_0, a_1-a_0, \ldots, a_n-a_0)$ with $0=a_0'<a_1'<\ldots <a_n'$. It follows immediately that $\mathcal{G}^{N}$ is complete, thereby ending the proof.
\end{proof}
 
\begin{remark}
We have seen in Theorem \ref{agama} in order to construct a large complete set we only need to first find a small complete set and then add terms of a well-balanced finite sequence into the set, thereby obtaining a complete set. This process requires adding negative integers. We can avoid the negative integers by examining the following result encapsulated in the following theorem. 
\end{remark}

\begin{theorem}\label{Agama1}
Let $\mathcal{F}$ be a finite set of integers. If $\mathcal{F}$ is a  homogenous arithmetic progression of odd length, then $\mathcal{F}$ is complete in $\mathbb{Z}$.
\end{theorem}

\begin{proof}
 Let us consider the homogenous arithmetic progression $\mathcal{F}=\{d, 2d, \ldots ,(n-1)d, nd\}$. Clearly we see that $(d\cdot 2d\cdots nd)=d^{n}n!$. We observe that \begin{align}d^{n}n!=\frac{n(n+1)}{2}d^{n}\bigg(2(n-2)!-4\frac{(n-2)!}{n+1}\bigg).\nonumber
\end{align}Suppose $\mathcal{F}$ is of odd length, then it is easy to see that $(n+1)|4(n-2)!$. Thus, $\bigg(2(n-2)!-4\frac{(n-2)!}{n+1}\bigg)\in \mathbb{Z}$, and it follows that $\mathcal{F}$ is complete in $\mathbb{Z}$, as required. 
\end{proof}

\begin{remark}
This result , albeit easy to state, is very useful for the theory. It helps us to construct complete sets of integers of any length we wish. More than this, it relates two important concepts of sets of integers, one of which is  widely studied in the whole of mathematics and has led to massive developments, arithmetic progression. It is also worth pointing out that the converse of Theorem \ref{Agama1} is not true, since there are complete sets that are not homogenous arithmetic progressions.
\end{remark}

\begin{example}
Theorem \ref{Agama1} informs us that the sets $\{3,6,9\}$, $\{23,46, 69, 92, 115\}$, $\{4,8, 12, 16, 20\}$, $\{7,14, 21,28,35\}$, $\{101,202,303, 404, 505\}$, $\{11, 22, 33,\}$, $\{9,18,27,$ $36,45,54,63,72,81\}$ are all complete in $\mathbb{Z}$.
\end{example}

\begin{corollary}
Every interval $[1,N]$ of positive integers of odd length $N$ is complete in $\mathbb{Z}$.
\end{corollary}

\begin{proof}
The result follows immediately from Theorem \ref{Agama1}, since all integers in the interval $[1,N]$ form a homogenous arithmetic progression.
\end{proof}

\begin{conjecture}
Let $\mathcal{H}:=\{p_1,p_2, \ldots, p_n\}$ be a set of odd length of the first $n$ odd primes, with $3=p_1<p_2\cdots <p_n$. Then either $\mathcal{H}$ is a complete set or \begin{align}\sum \limits_{i=1}^{n}p_i=L,\nonumber
\end{align}is prime or $\omega(L)=2$, where $\omega(L)=\sum \limits_{p|L}1$.
\end{conjecture}

\begin{conjecture}
Every finite set $\mathcal{T}\subset \mathbb{N}$ can be completed in $\mathbb{N}$.
\end{conjecture}

\section{THE NUMBER OF COMPLETE SETS IN $\mathbb{Z}$}
In this section we turn our attention to counting the number complete subsets that can be formed from any finite set of integers. We begin addressing the problem from a narrower perspective, which is to say we seek the maximum number of complete subsets of the set $\{1,2, \ldots, N\}$ of integers. We obtain a lower bound in the following results.

\begin{lemma}\label{harmonic}
The estimate \begin{align}\sum \limits_{n\leq x}\frac{1}{n}=\log x+O(1),\nonumber
\end{align}is valid.
\end{lemma}

\begin{proof}
See Theorem 6.9 in the book of Nathanson \cite{nathanson2000elementary}.
\end{proof}
\bigskip

\begin{theorem}\label{counting1}
Let $\mathcal{C}(N)$ denotes the total number of complete subsets of the set $\mathcal{H}=\{1,2,\ldots, N\}$ of integers, then \begin{align}\mathcal{C}(N)\gg N\log N.\nonumber
\end{align}
\end{theorem}  

\begin{proof}
Using Theorem \ref{Agama1}, we only need to count the number of homogenous arithmetic progression of length $k$ that can be formed from the interval $[1,N]$, where $k$ runs through the odd numbers no bigger than  $N$. Let us consider all the homogenous arithmetic progressions of length $3$ that can be formed from the interval $[1,N]$; clearly there are $\left \lfloor \frac{N}{3}\right \rfloor$ such number of sets. We have the  total count for those of length $5$  to be $\left \lfloor \frac{N}{5}\right \rfloor$. The total count for those of length $j$ is given by $\left \lfloor \frac{N}{j}\right \rfloor$. This culminates into the assertion that the total number of such complete sets is given by \begin{align}\sum \limits_{j\geq 3}\left \lfloor \frac{N}{j}\right \rfloor, \nonumber
\end{align} where $j$ runs over the odd numbers no bigger that $N$. Hence the assertion follows immediately by applying Lemma \ref{harmonic}.
\end{proof}

\begin{example}
Let us consider the interval $[1,10]$. The complete sets $\{1,2,3\}$, $\{2,4,6\}$ and $\{3,6,9\}$, represents complete sets of size $3$ that can be formed from the interval $[1,10]$. Clearly there are $3$ of them.  Again the set $\{1,2,3,4,5\}$ and $\{2,4,6,8,10\}$ represents complete sets of size $5$ that can be formed from the interval. Similarly, there is only one complete set of size $7$ that can be formed from $[1,10]$ and $\{1,2,3,4,5,6,7\}$ is an example. The set $\mathcal{L}=\{1,2,3,4,5,6,7,8,9\}$ represents a complete set of size $9$ that can be formed from $[1,10]$. Thus in total there are at least $7$ constructible complete sets that can be formed from the interval $[1,10]$.  
\end{example}

\begin{conjecture}\label{counting 2}
Let $\mathcal{C}(N)$ be the number of complete subsets of the set $\mathcal{H}:=\{1,2,\ldots, N\}$. Then \begin{align}\mathcal{C}(N)=O(N(\log N)(\log \log N)).\nonumber
\end{align}
\end{conjecture}
\bigskip

\begin{remark}
Conjecture \ref{counting 2} tells us a great deal about the distribution of complete subsets of the set $\{1,2\ldots, N\}$. In probabilistic language, it tells us that the chance of any subset of the set $\{1,2,\ldots, N\}$ to be complete is very minimal, since \begin{align}\lim \limits_{N\longrightarrow \infty}\frac{N(\log N)(\log \log N)}{2^{N}}=0. \nonumber
\end{align}
\end{remark}
\section{END REMARKS}
As mentioned earlier, there are some complete sets that are not homogeneous arithmetic progressions. Given the interval $[1, 10]$ it turns out that the sets $\{3,5,7\}$, $\{2,5,7\}$ and $\{2,3,5\}$, that were not taken into account in Theorem \ref{counting1}, are also complete in $\mathbb{Z}$. Such a loss becomes very significant an $N$ is taken sufficiently large. This significant loss indicates something wierd unfolding as $N$ increases without bound, and does suggest the lower bound $\mathcal{C}(N)\gg N\log N$ is not the best possible and can be improved. To this end we raise some questions whose answer may be attributed to such a loss.

\begin{question}
Does there exist complete sets of the form $\{r,r^2,r^3,\ldots, r^n\}$?
\end{question}

\begin{question}
If the set $\mathcal{F}$ is complete, does there exist some integer $s<M$ such that $\mathcal{F}+\{s\}$ is complete?
\end{question}

%\begin{table}[ht]
%\caption{}\label{eqtable}
%\renewcommand\arraystretch{1.5}
%\noindent\[
%\begin{array}{|c|c|c|}
%\hline
%&{-\infty}&{+\infty}\\
%\hline
%{f_+(x,k)}&e^{\sqrt{-1}kx}+s_{12}(k)e^{-\sqrt{-1}kx}&s_{11}(k)e^
%{\sqrt{-1}kx}\\
%\hline
%{f_-(x,k)}&s_{22}(k)e^{-\sqrt{-1}kx}&e^{-\sqrt{-1}kx}+s_{21}(k)e^{\sqrt
%{-1}kx}\\
%\hline
%\end{array}
%\]
%\end{table}

%\begin{figure}[tb]
%\blankbox{.6\columnwidth}{5pc}
%\includegraphics{geometric proof.pdf}
%\caption{This is an example of a figure caption with text.}
%\label{firstfig}
%\end{figure}

%\begin{figure}[tb]
%\blankbox{.75\columnwidth}{3pc}
%\includegraphics{lion.png}
%\caption{}\label{otherfig}
%\end{figure}

\bibliographystyle{amsplain}

\end{document}